\documentclass{article}

\usepackage{amsfonts}
\usepackage{amsthm}
\usepackage{amsmath}
\usepackage{amssymb}
\usepackage{mathtools} %for prescript
\usepackage{authblk} %for affil
\usepackage{xcolor}
\usepackage{soul} %for strikethrough
\usepackage{mathrsfs}
\usepackage{algorithm}
\usepackage{algpseudocode}
\usepackage[top=1in, bottom=1in, left=1in, right=1in]{geometry}
\newcommand{\norm}[1]{\left \lVert#1\right \rVert}
\newtheorem{theorem}{Theorem}[section]
\newtheorem{lemma}{Lemma}[section]
\newtheorem{corollary}{Corollary}[section]

\theoremstyle{definition}
\newtheorem{definition}{Definition}[section]

\theoremstyle{remark}
\newtheorem{remark}{Remark}[section]

\numberwithin{equation}{section}

\begin{document}
	
	%		\begin{frontmatter}

		\title{Perturbation of fractional strongly continuous cosine family operators}
		
			%%this line removes the date, but space is still left for it;
		%if used, remove the \vspace{-1cm}
		\date{}
		
		%this gives the date in the form Mon 30 Jan 2012, 8:57pm;
		%if used, retain the \vspace{-1cm}
		%\date{\shortdayofweekname{\day}{\month}{\year}{ }\mydate\today}

		%\author[]{}
		\author[]{Ismail T. Huseynov \thanks{Email: 
				\texttt{ismail.huseynov@emu.edu.tr}}}
		\author[]{Arzu Ahmadova \thanks{Email: 
				\texttt{arzu.ahmadova@emu.edu.tr}}}
		\author[]{Nazim I. Mahmudov\thanks{Corresponding author. Email: \texttt{nazim.mahmudov@emu.edu.tr}}}

		\affil{Department of Mathematics, Eastern Mediterranean University, Mersin 10, 99628, T.R. North Cyprus, Turkey}
		
		% Latex won't make the title unless told:
		\maketitle
		
		%%to remove the space left for date, use:
		\begin{abstract}
			\noindent Perturbation theory has long been a very useful tool in the hands of mathematicians and physicists. The purpose of this paper is to prove some perturbation results for infinitesimal generators of fractional strongly continuous cosine families. That is, we impose sufficient conditions such that $\mathscr{A}$ is the infinitesimal generator of a fractional strongly continuous cosine family in a Banach space $\mathcal{X}$, and $\mathscr{B}$ is a bounded linear operator in $\mathcal{X}$, then $\mathscr{A}+\mathscr{B}$ is also the infinitesimal generator of a fractional strongly continuous cosine family in $\mathcal{X}$. Our results coincide with the classical ones when $\alpha=2$.
			
		\end{abstract}

	%\end{frontmatter}
	\section{Introduction}\label{Sec:intro}
	A strong inspiration for the study of perturbation theory for strongly continuous operator families comes from the fact that they have proven to be useful tools for evolution equations in modeling many physical phenomena. The perturbation theory of linear operators in a Banach space has been explored to a considerable extent, most notably by Phillips \cite{philips},  Travis \& Webb \cite{TW3} and Lutz \cite{lutz}.
	In \cite{philips}, Phillips intended to perturb the infinitesimal generator by adding to it a linear bounded operator and investigating some perturbation results for infinitesimal generators of strongly continuous $C_{0}$-semigroups.
	In \cite{TW3}, Travis and Webb have proposed sufficient conditions for perturbed cosine operator families.
In \cite{lutz}, Lutz has first studied the implication for the homogeneous initial value problem associated with the infinitesimal generator of the cosine function generated by bounded time-varying perturbations:
	
	\begin{equation}\label{secondorder}
		\begin{cases}
			v^{\prime \prime}(t)= (\mathscr{A}+\mathscr{B}(t))v(t), \quad t \in \mathbb{R},\\
			v(0)=v_0 \in \mathcal{D}(\mathscr{A}), \quad v^{\prime}(0)=v_1 \in \mathcal{D}(\mathscr{A}),
		\end{cases}
	\end{equation}  
where $\mathscr{B}(\cdot):\mathbb{R_{+}}\to \mathcal{L}(\mathcal{X})$ is a continuously differentiable function on $\mathbb{R_{+}}$.
	
	The fractional analogue of the abstract problem \eqref{secondorder} has been established by Bazhlekova in \cite{bazhlekova2}. Bazhlekova \cite{bazhlekova2} has proposed to uniquely determine a classical solution of the following homogeneous abstract Cauchy problem for the fractional evolution equation with Caputo derivative by time-dependent perturbations:
	\begin{equation*}
		\begin{cases*}
			\left( \prescript{C}{}{\mathscr{D}^{\alpha}_{0_{+}}}v\right) (t)= (\mathscr{A}+\mathscr{B}(t))v(t), \quad t > 0,\\
			v(0)=v_0 \in \mathcal{D}(\mathscr{A}), \quad v^{\prime}(0)=0,
		\end{cases*}
	\end{equation*} 
	where $\mathscr{B}(\cdot):\mathbb{R_{+}}\to \mathcal{L}(\mathcal{X})$ is a continuous function on $\mathbb{R_{+}}$.

	Furthermore, the fractional analogue of the inhomogeneous abstract problem \eqref{secondorder} was developed by Ahmadova et al. in \cite{A-H-M} and the results presented in \cite{A-H-M} extend those of \cite{lutz,bazhlekova2} in several aspects. The authors in \cite{A-H-M} have shown that the inhomogeneous abstract Cauchy problem with the infinitesimal generator $\mathscr{A}$ of fractional cosine families remains uniformly well-posed under bounded time-varying perturbations:
	\begin{equation*}
		\begin{cases*}
			\left( \prescript{C}{}{\mathscr{D}^{\alpha}_{0_{+}}}v\right) (t)= (\mathscr{A}+\mathscr{B}(t))v(t)+h(t), \quad t > 0,\\
			v(0)=v_0 \in \mathcal{D}(\mathscr{A}), \quad v^{\prime}(0)=v_1 \in \mathcal{D}(\mathscr{A}),
		\end{cases*}
	\end{equation*} 
	where $\mathscr{B}(\cdot):\mathbb{R_{+}}\to \mathcal{L}(\mathcal{X})$ and $h(\cdot):\mathbb{R_{+}}\to \mathcal{X}$ are continuously differentiable functions on $\mathbb{R_{+}}$.

	Therefore, the goal of this paper is to develop some perturbation results for infinitesimal generators of fractional strongly continuous cosine families.
	That is, we establish sufficient conditions such that $\mathscr{A}$ is the infinitesimal generator of a fractional strongly continuous cosine (sine) family in a Banach space $\mathcal{X}$, and $\mathscr{B}$ is a bounded linear operator in $\mathcal{X}$, then $\mathscr{A}+\mathscr{B}$ is also the infinitesimal generator of a fractional strongly continuous cosine (sine) family in $\mathcal{X}$. The pioneering work on this subject in the classical sense was done by Phillips \cite{philips} and Travis \& Webb \cite{TW3} and our development follows these approaches.

	\section{Preliminaries}\label{prel}
	We start this section by briefly introducing the essential structure of an operator theory for linear operators. For the more salient details on these subjects, see the textbook \cite{engel}.

	Let $\mathbb{R_{+}}=[0,\infty)$ and $\mathbb{N}$ denote the set of natural numbers with $\mathbb{N}_{0}=\mathbb{N}\cup\left\lbrace 0\right\rbrace$.
	Let $\mathcal{X}$ be a Banach space equipped with the norm $\norm{\cdot}$. 
	We donote by $\mathcal{L}(\mathcal{X})$ the Banach algebra of all bounded linear operators on $\mathcal{X}$ and becomes a Banach space with regard to the norm $\norm{\mathscr{T}}=\sup\left\lbrace \norm{\mathscr{T}x}: \norm{x} \leq 1\right\rbrace $, for any $\mathscr{T}\in \mathcal{L}(\mathcal{X})$.
	Let $\mathcal{D}(\mathscr{A})$ be the domain of $\mathscr{A}$ and $\rho(\mathscr{A})$ be the resolvent set of $\mathscr{A}$. The identity and zero operators on $\mathcal{X}$ are denoted by $\mathscr{I}\in \mathcal{L}(\mathcal{X})$ and $0\in \mathcal{L}(\mathcal{X})$, respectively.

	We will use the following functional spaces \cite{engel} through the paper:
	
	\begin{itemize}
		
		\item
		$\mathbb{C}\left(\mathbb{R_{+}},\mathcal{X}\right)$ denotes the Banach space of continuous $\mathcal{X}$-valued functions $g(\cdot):\mathbb{R_{+}}\to \mathcal{X} $ equipped with an infinity norm $\|g\|_{\infty}=\sup\limits_{t\in \mathbb{R_{+}}}\norm{g(t)}$;

		\item
		$\mathbb{C}^{n}(\mathbb{R_{+}},\mathcal{X})$, $n\in \mathbb{N}$ denotes the Banach space of $n$-times continuously differentiable $\mathcal{X}$-valued functions defined by 
		\begin{equation*}
			\mathbb{C}^{n}(\mathbb{R_{+}},\mathcal{X})=\left\lbrace g(\cdot)\in\mathbb{C}^{n}(\mathbb{R_{+}},\mathcal{X}): g^{(n)}(\cdot) \in\mathbb{C}(\mathbb{R_{+}},\mathcal{X}), n \in \mathbb{N}\right\rbrace
		\end{equation*}
		and equipped with an infinity norm $\|g\|_{\infty}=\sum\limits_{k=0}^{n}\sup\limits_{t\in \mathbb{R_{+}}}\norm{g^{(k)}(t)}$.
		In addition, $\mathbb{C}^{n}(\mathbb{R_{+}},\mathcal{X})\subset \mathbb{C}(\mathbb{R_{+}},
		\mathcal{X})$, $n\in \mathbb{N}$.
	\end{itemize}

 We will use the following notation for $\alpha\geq 0$:
		
		\begin{align*}
			g_{\alpha}(t)=
			\begin{cases*}
				\frac{t^{\alpha-1}}{\Gamma(\alpha)}, \quad t>0,\\
				0, \qquad t \leq 0,
			\end{cases*}
		\end{align*}
		where $\Gamma(\cdot):\mathbb{R_{+}}\to \mathbb{R}$ is the gamma function. Note that $g_{0}(t)=0$, since $\frac{1}{\Gamma(0)}=0$, $t\in \mathbb{R}$. These
		functions are satisfying the following semigroup property:
		\begin{equation}\label{galpha}
			(g_{\alpha}\ast g_{\beta})(t)=g_{\alpha+\beta}(t), \quad t \in \mathbb{R}.
		\end{equation}

	\begin{definition}\cite{kexue1}
		Let $\alpha \in (1,2)$. A family $\mathscr{C}_{\alpha}(\cdot; \mathscr{A}):\mathbb{R_{+}}\to \mathcal{L}(\mathcal{X})$ of all bounded linear operators on $\mathcal{X}$ is called a fractional strongly continuous cosine family if it satisfies the following hypotheses:
		\begin{itemize}
			\item $\mathscr{C}_{\alpha}(t;\mathscr{A})$ is strongly continuous for all $t \in \mathbb{R_{+}}$ and $\mathscr{C}_{\alpha}(0;\mathscr{A})=\mathscr{I}$;
			\item $\mathscr{C}_{\alpha}(s;\mathscr{A})\mathscr{C}_{\alpha}(t;\mathscr{A})=\mathscr{C}_{\alpha}(t;\mathscr{A})\mathscr{C}_{\alpha}(s;\mathscr{A})$ for all $s,t\in \mathbb{R_{+}}$;
			\item The functional equation
			\begin{equation}
				\mathscr{C}_{\alpha}(s;\mathscr{A})\mathcal{I}_{t}^{\alpha}\mathscr{C}_{\alpha}(t;\mathscr{A})-\mathcal{I}_{s}^{\alpha}\mathscr{C}_{\alpha}(s;\mathscr{A})\mathscr{C}_{\alpha}(t;\mathscr{A})=\mathcal{I}_{t}^{\alpha}\mathscr{C}_{\alpha}(t;\mathscr{A})-\mathcal{I}_{s}^{\alpha}\mathscr{C}_{\alpha}(s;\mathscr{A}) \quad \text{holds for all} \quad s,t \in \mathbb{R_{+}}.
			\end{equation}
		\end{itemize}
	\end{definition}
	The closed linear operator $\mathscr{A}$ is defined by
	\begin{align}
		\mathscr{A}x \coloneqq \Gamma(\alpha+1)\lim_{t \to 0_{+}}\frac{\mathscr{C}_{\alpha}(t;\mathscr{A})x-x}{t^{\alpha}}, \quad x \in \mathcal{D}(\mathscr{A})\coloneqq \left\lbrace x \in \mathcal{X}: \mathscr{C}_{\alpha}(\cdot;\mathscr{A}) \in \mathbb{C}^{2}\left( \mathbb{R_{+}},\mathcal{X}\right) \right\rbrace,
	\end{align}
	where $\mathscr{A}$ is the infinitesimal generator of the fractional strongly continuous cosine family $\{ \mathscr{C}_{\alpha}(t;\mathscr{A}):t\in \mathbb{R_{+}}\} $.
	
	\begin{definition}\cite{kexue1}
		A fractional strongly continuous cosine family $\{ \mathscr{C}_{\alpha}(t;\mathscr{A}):t\in \mathbb{R_{+}}\} $ is said to be exponentially bounded if there exists constants $M\geq 1$, $\omega \geq 0$ such that
		\begin{equation}\label{exp}
			\norm{\mathscr{C}_{\alpha}(t;\mathscr{A})}\leq M e^{\omega t}, \quad  t \in \mathbb{R_{+}}. 
		\end{equation} 
	\end{definition}
	
	\begin{definition}\cite{kexue1}
		The fractional strongly continuous  sine family $\mathscr{S}_{\alpha}(\cdot;\mathscr{A}):\mathbb{R_{+}}\to \mathcal{L}(\mathcal{X})$ associated with $\mathscr{C}_{\alpha}$ is defined by
		\begin{equation}
			\mathscr{S}_{\alpha}(t;\mathscr{A})x=\int_{0}^{t}\mathscr{C}_{\alpha}(s;\mathscr{A})x\mathrm{d}s, \quad x \in \mathcal{X}, \quad t \in \mathbb{R_{+}}.
		\end{equation}
		This implies that 
		\begin{equation}\label{AA}
			\mathscr{D}^{1}_{t}\mathscr{S}_{\alpha}(t;\mathscr{A})x=\mathscr{C}_{\alpha}(t;\mathscr{A})x, \quad x \in \mathcal{X}, \quad t \in \mathbb{R_{+}}.
		\end{equation}
	\end{definition}

	\begin{definition}\cite{kexue1}
		The fractional strongly continuous  Riemann-Liouville family $\mathscr{T}_{\alpha}(\cdot;\mathscr{A}):\mathbb{R_{+}}\to \mathcal{L}(\mathcal{X})$ associated with $\mathscr{C}_{\alpha}$ is defined by
		\begin{equation}\label{Talpha}
			\mathscr{T}_{\alpha}(t;\mathscr{A})x=\mathcal{I}_{t}^{\alpha-1}\mathscr{C}_{\alpha}(t;\mathscr{A})x=\int_{0}^{t}g_{\alpha-1}(t-s)\mathscr{C}_{\alpha}(s;\mathscr{A})x\mathrm{d}s,\quad x \in \mathcal{X}, \quad t \in \mathbb{R_{+}}.
		\end{equation}
	\end{definition}

	\begin{theorem}[\cite{kexue1}]
		Let $\{ \mathscr{C}_{\alpha}(t;\mathscr{A}):t\in \mathbb{R_{+}}\} $ be a fractional strongly continuous cosine family in $\mathcal{X}$ satisfying \eqref{exp}  and let $\mathscr{A}$ be the infinitesimal generator of $\{ \mathscr{C}_{\alpha}(t;\mathscr{A}):t\in \mathbb{R_{+}}\}$. Then for $Re(\lambda)>\omega$, $\lambda^{\alpha} \in \rho(\mathscr{A})$, $\alpha \in (1,2]$, the following relations hold true: 
		\begin{align}\label{kosinus}
			\lambda^{\alpha-1} &\mathcal{R}(\lambda^
			{\alpha};\mathscr{A})x=\int_{0}^{\infty}e^{-\lambda t}\mathscr{C}_{\alpha}(t;\mathscr{A}
			)x\mathrm{d}t, \quad x \in \mathcal{X},\\
			 \lambda^{\alpha-2}&\mathcal{R}(\lambda^{\alpha};\mathscr{A})x=\int_{0}^{\infty}e^{-\lambda t}\mathscr{S}_{\alpha}(t;\mathscr{A})x\mathrm{d}t, \quad x \in \mathcal{X},\\
			&\mathcal{R}(\lambda^{\alpha};\mathscr{A})x=\int_{0}^{\infty}e^{-\lambda t}\mathscr{T}_{\alpha}(t;\mathscr{A})x\mathrm{d}t, \quad x \in \mathcal{X}\label{sinus}.
		\end{align} 
	\end{theorem}

	\section{Main results}\label{pertub}
	In many concrete situations, the fractional evolution equation is given as a sum of several terms that have various physical meanings and various mathematical properties. While the mathematical analysis for each term may be straightforward, it is not entirely clear what happens after the summation. In the context of perturbed generators of fractional cosine families, we take this as a starting point. To study some perturbation results for fractional strongly continuous cosine families, we first prove the following auxiliary lemma, which plays an important role in the proof of Theorem \ref{thm1}, and this lemma establishes a connection between the resolvents of cosine (sine) families generated by $\mathscr{A}$ and $\mathscr{A}+\mathscr{B}$, respectively.
	\begin{lemma}\label{lem}
		Let $\mathscr{A}$ be a closed linear operator on $\mathcal{X}$ to $\mathcal{X}$ and suppose $\mathscr{B}\in \mathcal{L}(\mathcal{X})$ is such that $\norm{\mathscr{B}\mathcal{R}( \lambda^{\alpha};\mathscr{A})} =\theta<1$ for some $\lambda^{\alpha}\in \rho(\mathscr{A})$, $\alpha \in (1,2]$. Then, $\mathscr{A}+\mathscr{B}$ is a closed linear operator with domain $\mathcal{D}(\mathscr{A})$ and $\mathcal{R}( \lambda^{\alpha};\mathscr{A}+\mathscr{B})$ exists and 
			\begin{equation}\label{star}
				\mathcal{R}( \lambda^{\alpha};\mathscr{A}+\mathscr{B})=\sum_{n=0}^{\infty}\mathcal{R}( \lambda^{\alpha};\mathscr{A})\Big[\mathscr{B}\mathcal{R}( \lambda^{\alpha};\mathscr{A})\Big]^{n}.
			\end{equation}
			
			Furthermore, 
			\begin{equation}
				\norm{\mathcal{R}( \lambda^{\alpha};\mathscr{A}+\mathscr{B})-\mathcal{R}( \lambda^{\alpha};\mathscr{A})} \leq \norm{\mathcal{R}( \lambda^{\alpha};\mathscr{A})}\theta (1-\theta)^{-1}.
			\end{equation}
		\end{lemma}
	
\begin{proof}
	Obviously, $\mathscr{A}+\mathscr{B}$ is a closed linear operator with domain $\mathcal{D}(\mathscr{A})$. Since $\norm{\mathscr{B}\mathcal{R}( \lambda^{\alpha};\mathscr{A})} =\theta<1$, we note that
	\begin{equation*}
		\mathcal{R}\equiv\sum_{n=0}^{\infty}\mathcal{R}( \lambda^{\alpha};\mathscr{A})\Big[\mathscr{B}\mathcal{R}( \lambda^{\alpha};\mathscr{A})\Big]^{n}=\mathcal{R}( \lambda^{\alpha};\mathscr{A})\Big[\mathscr{I}-\mathscr{B}\mathcal{R}( \lambda^{\alpha};\mathscr{A})\Big]^{-1},
	\end{equation*}
and hence that
\begin{equation*}
	\Big[\lambda^{\alpha}\mathscr{I}-(\mathscr{A}+\mathscr{B})\Big]\mathcal{R}=\Big[\mathscr{I}-\mathscr{B}\mathcal{R}( \lambda^{\alpha};\mathscr{A})\Big]\Big[\mathscr{I}-\mathscr{B}\mathcal{R}( \lambda^{\alpha};\mathscr{A})\Big]^{-1}=\mathscr{I}.
\end{equation*}
Moreover, the range of $\mathcal{R}$ is precisely $\mathcal{D}(\mathscr{A})$ since the range of $\Big[\mathscr{I}-\mathscr{B}\mathcal{R}( \lambda^{\alpha};\mathscr{A})\Big]^{-1}$ is $\mathcal{X}$.
Therefore, given $x\in \mathcal{D}(\mathscr{A})$ there exists a $y$ such that $x=\mathcal{R}y$.
Therefore, we attain that
\begin{equation*}
\mathcal{R}\Big[\lambda^{\alpha}\mathscr{I}-(\mathscr{A}+\mathscr{B})\Big] x=\mathcal{R}\Big[\lambda^{\alpha}\mathscr{I}-(\mathscr{A}+\mathscr{B})\Big]\mathcal{R}y=\mathcal{R}y=x,
\end{equation*}
so that $\mathcal{R}$ is both a left and a right inverse of $\Big[\lambda^{\alpha}\mathscr{I}-(\mathscr{A}+\mathscr{B})\Big]$. 
The bound $\norm{\mathcal{R}( \lambda^{\alpha};\mathscr{A}+\mathscr{B})-\mathcal{R}( \lambda^{\alpha};\mathscr{A})} $ comes directly from the expansion \eqref{star}:
\small{
\begin{align*}
	\norm{\mathcal{R}( \lambda^{\alpha};\mathscr{A}+\mathscr{B})-\mathcal{R}( \lambda^{\alpha};\mathscr{A})}&=\norm{\sum_{n=1}^{\infty}\mathcal{R}( \lambda^{\alpha};\mathscr{A})\Big[\mathscr{B}\mathcal{R}( \lambda^{\alpha};\mathscr{A})\Big]^{n}}\\&\leq \norm{\mathcal{R}( \lambda^{\alpha};\mathscr{A})}\sum_{n=0}^{\infty}\norm{\mathscr{B}\mathcal{R}( \lambda^{\alpha};\mathscr{A})}^{n}\norm{\mathscr{B}\mathcal{R}( \lambda^{\alpha};\mathscr{A})}\\
	 &\leq \norm{\mathbb{R}( \lambda^{\alpha};\mathscr{A})}\Big(1-\norm{\mathscr{B}\mathcal{R}( \lambda^{\alpha};\mathscr{A})}\Big)^{-1}\norm{\mathscr{B}\mathcal{R}( \lambda^{\alpha};\mathscr{A})}\\&=\norm{\mathcal{R}( \lambda^{\alpha};\mathscr{A})} \theta (1-\theta)^{-1}.
\end{align*}
}
\end{proof}
\begin{corollary}\label{corol}
	Let $\mathscr{A}$ be a closed linear operator on $\mathcal{X}$ to $\mathcal{X}$ and suppose $\mathscr{B}\in \mathcal{L}(\mathcal{X})$ is such that $\norm{\mathscr{B}\mathcal{R}( \lambda^{\alpha};\mathscr{A})} =\theta<1$ for some $\lambda^{\alpha}\in \rho(\mathscr{A})$, $\alpha \in (1,2]$. Then, $\mathscr{A}+\mathscr{B}$ is a closed linear operator with domain $\mathcal{D}(\mathscr{A})$ and $\lambda^{\alpha-1}\mathcal{R}( \lambda^{\alpha};\mathscr{A}+\mathscr{B})$ exist and 
\begin{equation}\label{star-1}
	\lambda^{\alpha-1}\mathcal{R}( \lambda^{\alpha};\mathscr{A}+\mathscr{B})=\sum_{n=0}^{\infty}\lambda^{\alpha-1}\mathcal{R}( \lambda^{\alpha};\mathscr{A})\Big[\mathscr{B}\mathcal{R}( \lambda^{\alpha};\mathscr{A})\Big]^{n}.
\end{equation}

Furthermore, 
\begin{equation}
	\norm{\lambda^{\alpha-1}\mathcal{R}( \lambda^{\alpha};\mathscr{A}+\mathscr{B})-\lambda^{\alpha-1}\mathcal{R}( \lambda^{\alpha};\mathscr{A})} \leq \norm{\lambda^{\alpha-1}\mathcal{R}( \lambda^{\alpha};\mathscr{A})}\theta (1-\theta)^{-1}.
\end{equation}	
\end{corollary}

In the following theorem, we impose sufficient conditions such that $\mathscr{A}$ is the infinitesimal generator of a fractional strongly continuous cosine (sine) family in $\mathcal{X}$, and $\mathscr{B}$ is a bounded linear operator in $\mathcal{X}$, then $\mathscr{A}+\mathscr{B}$ is also the infinitesimal generator of a fractional strongly continuous cosine (sine) family in $\mathcal{X}$.	

	\begin{theorem}\label{thm1}
		Let $\mathscr{A}$ be the infinitesimal generator of the fractional strongly continuous cosine family $\left\lbrace \mathscr{C}_{\alpha}(t;\mathscr{A}): t \in \mathbb{R_{+}}\right\rbrace $ with  \eqref{exp}, and let $\left\lbrace \mathscr{S}_{\alpha}(t;\mathscr{A}): t \in \mathbb{R_{+}}\right\rbrace $ denote the fractional strongly continuous sine family associated with $\mathscr{C}_{\alpha}$ where $\alpha \in (1,2]$. Furthermore, suppose that $\mathscr{B}\in \mathcal{L}(\mathcal{X})$.  Then, fractional strongly continuous families of linear bounded operators $\mathscr{C}_{\alpha}(\cdot;\mathscr{A}+\mathscr{B}),\mathscr{S}_{\alpha}(\cdot;\mathscr{A}+\mathscr{B})\in \mathcal{L}(\mathcal{X})$ generated by $\mathscr{A}+\mathscr{B}$ (defined on $\mathcal{D}(\mathscr{A})$) can be represented by the series expansion:
			\begin{align}
				&\mathscr{C}_{\alpha}(t;\mathscr{A}+\mathscr{B})x\coloneqq \sum_{n=0}^{\infty}\mathscr{C}_{\alpha,n}(t;\mathscr{A})x, \quad x \in \mathcal{X},\label{Cdefine}\\
				&\mathscr{S}_{\alpha}(t;\mathscr{A}+\mathscr{B})x\coloneqq \sum_{n=0}^{\infty}\mathscr{S}_{\alpha,n}(t;\mathscr{A})x, \quad x\in \mathcal{X}\label{Sdefine},
			\end{align}
			where for $t\in \mathbb{R_{+}}$, $x\in \mathcal{X}$ and $n \in \mathbb{N}$,  
			\begin{align}
				&\mathscr{C}_{\alpha,0}(t;\mathscr{A})x\coloneqq \mathscr{C}_{\alpha}(t;\mathscr{A})x,\quad \mathscr{S}_{\alpha,0}(t;\mathscr{A})x\coloneqq \mathscr{S}_{\alpha}(t;\mathscr{A})x,\nonumber\\
				&\mathscr{C}_{\alpha,n}(t;\mathscr{A})x\coloneqq \int_{0}^{t}\mathscr{C}_{\alpha}(t-s;\mathscr{A})\mathscr{B}\mathscr{S}_{\alpha,n-1}(s;\mathscr{A})x\mathrm{d}s,\label{Calphan}\\
				&\mathscr{S}_{\alpha,n}(t;\mathscr{A})x\coloneqq \int_{0}^{t}\mathscr{T}_{\alpha}(t-s;\mathscr{A})\mathscr{B}\mathscr{S}_{\alpha,n-1}(s;\mathscr{A})x\mathrm{d}s\label{Salphan}.
			\end{align}
	\end{theorem}
	\begin{proof}
	 From \eqref{exp}, it follows that $\norm{\mathscr{S}_{\alpha}(t;\mathscr{A})}\leq Mte^{\omega t}$. This implies that $\mathscr{C}_{\alpha}(t;\mathscr{A})$ and $\mathscr{S}_{\alpha}(t;\mathscr{A})$ are strongly continuous on $\mathbb{R_{+}}$. Moreover, by using the formula \eqref{Talpha}, we obtain: 
	 \begin{equation}\label{A}
	 	\norm{\mathscr{T}_{\alpha}(t;\mathscr{A})}\leq \int_{0}^{t}g_{\alpha-1}(t-s)\norm{\mathscr{C}_{\alpha}(s;\mathscr{A})}\mathrm{d}s\leq Me^{\omega t}\int_{0}^{t}\frac{(t-s)^{\alpha-2}}{\Gamma(\alpha-1)}\mathrm{d}s=Me^{\omega t}g_{\alpha}(t), \quad t\in \mathbb{R_{+}}.
	 \end{equation}
	 
	  Then, we suppose  $\mathscr{S}_{\alpha,n}(t;\mathscr{A})$ and $\mathscr{C}_{\alpha,n}(t;\mathscr{A})$ are strongly continuous on $\mathbb{R_{+}}$ and it is true that for $n \in N_{0}$:
	 
	 \begin{align}\label{formula-1}
	 	&\norm{\mathscr{S}_{\alpha,n}(t;\mathscr{A})}\leq M^{n+1} \norm{\mathscr{B}}^{n}e^{\omega t}g_{n\alpha+2}(t),\\
	 	&\norm{\mathscr{C}_{\alpha,n}(t;\mathscr{A})}\leq M^{n+1} \norm{\mathscr{B}}^{n}e^{\omega t}g_{(n-1)\alpha+3}(t).
	 	\label{formula-2}
	 \end{align}

		Firstly, we start with the formula \eqref{formula-1}. Then, $\mathscr{T}_{\alpha}(t-s;\mathscr{A})\mathscr{B}\mathscr{S}_{\alpha,n}(s;\mathscr{A})$ will be strongly continuous on $[0,t]$ such that the integral defining $\mathscr{S}_{\alpha,n+1}(t;\mathscr{A})$ exists in the strong topology. Moreover, 	 
			in the case of \eqref{formula-1}, this is true by our remark above for $n=0$. By mathematical induction principle, we verify \cite{A-H-M}:
			
			\begin{align}
				\norm{\mathscr{S}_{\alpha,n+1}(t;\mathscr{A})}&\leq \int_{0}^{t}\norm{\mathscr{T}_{\alpha}(t-s;\mathscr{A})}\norm{\mathscr{B}}\norm{\mathscr{S}_{\alpha,n}(s;\mathscr{A})} \mathrm{d}s\nonumber\\
				&\leq M^{n+2}\norm{\mathscr{B}}^{n+1}\int_{0}^{t}g_{\alpha}(t-s)
				e^{w(t-s)}e^{\omega s}g_{n\alpha+2}(s)\mathrm{d}s\nonumber\\
				&=M^{n+2}\norm{\mathscr{B}}^{n+1}e^{\omega t}g_{(n+1)\alpha+2}(t), \quad t \in \mathbb{R_{+}}.
			\end{align}
			
			The analogues procedure gives the bound for $\mathscr{C}_{\alpha,n}(t;\mathscr{A})$, $t\in \mathbb{R_{+}}$, as follows:
			\allowdisplaybreaks
			\begin{align}
				\norm{\mathscr{C}_{\alpha,n+1}(t;\mathscr{A})}&\leq \int_{0}^{t}\norm{\mathscr{C}_{\alpha}(t-s;\mathscr{A})}\norm{\mathscr{B}}\norm{\mathscr{S}_{\alpha,n}(s;\mathscr{A})} \mathrm{d}s\nonumber\\
				&\leq M^{n+2}\norm{\mathscr{B}}^{n+1}\int_{0}^{t}
				e^{w(t-s)}e^{\omega s}g_{n\alpha+2}(s)\mathrm{d}s\nonumber\\
				&=M^{n+2}\norm{\mathscr{B}}^{n+1}e^{\omega t}g_{n\alpha+3}(t),  \quad t \in \mathbb{R_{+}}.
			\end{align}
	Finally, for $t_1<t_2$, we have
	\begin{align}\label{first}
		\norm{\mathscr{S}_{\alpha,n+1}(t_2;\mathscr{A})x-\mathscr{S}_{\alpha,n+1}(t_1;\mathscr{A})x}&\leq \int_{0}^{t_1}\norm{\Big[\mathscr{T}_{\alpha}(t_2-s)-\mathscr{T}_{\alpha}(t_1-s)\Big]\mathscr{B}\mathscr{S}_{\alpha,n}(s;\mathscr{A})x}\mathrm{d}s\nonumber\\
		&+\int_{t_1}^{t_2}\norm{\mathscr{T}_{\alpha}(t_2-s)}\norm{\mathscr{B}}\norm{\mathscr{S}_{\alpha,n}(s;\mathscr{A})x}\mathrm{d}s, \quad x \in \mathcal{X},
	\end{align}	

	\begin{align}\label{second}
		\norm{\mathscr{C}_{\alpha,n+1}(t_2;\mathscr{A})x-\mathscr{C}_{\alpha,n+1}(t_1;\mathscr{A})x}&\leq \int_{0}^{t_1}\norm{\Big[\mathscr{C}_{\alpha}(t_2-s)-\mathscr{C}_{\alpha}(t_1-s)\Big]\mathscr{B}\mathscr{S}_{\alpha,n}(s;\mathscr{A})x}\mathrm{d}s\nonumber\\
		&+\int_{t_1}^{t_2}\norm{\mathscr{C}_{\alpha}(t_2-s)}\norm{\mathscr{B}}\norm{\mathscr{S}_{\alpha,n}(s;\mathscr{A})x}\mathrm{d}s, \quad x \in \mathcal{X}.
	\end{align}
As $t_1,t_2 \to t_0$, the integrands in the first terms on the right of \eqref{first},\eqref{second} to zero boundedly and the integrands of the second terms of \eqref{first},\eqref{second} are bounded.  Therefore, $\mathscr{S}_{\alpha,n+1}(t;\mathscr{A})x$, $\mathscr{C}_{\alpha,n+1}(t;\mathscr{A})x$ are strongly continuous on $\mathbb{R_{+}}$ for each fixed $x\in \mathcal{X}$. By induction, $\mathscr{S}_{\alpha,n}(t;\mathscr{A})x$, $\mathscr{C}_{\alpha,n}(t;\mathscr{A})x$ are well-defined, strongly continuous, and satisfying \eqref{formula-1} and \eqref{formula-2}, respectively, for $n\in \mathbb{N}_{0}$.   Hence, 
$\mathscr{S}_{\alpha}(t;\mathscr{A}+\mathscr{B})$, $\mathscr{C}_{\alpha}(t;\mathscr{A}+\mathscr{B})$ are uniformly convergent on compact subsets of $\mathbb{R_{+}}$ with respect to the operator norm topology. Moreover, the series are majorized by the following series expansions \cite{A-H-M}:

			\begin{align}\label{3}
				&\norm{\mathscr{S}_{\alpha}(t;\mathscr{A}+\mathscr{B})} \leq Me^{\omega t}tE_{\alpha,2}\left(M\norm{\mathscr{B}}t^{\alpha} \right), \quad t\in \mathbb{R_{+}},\\
				&\norm{\mathscr{C}_{\alpha}(t;\mathscr{A}+\mathscr{B})} \leq Me^{\omega t}t^{2-\alpha}E_{\alpha,3-\alpha}\left(M\norm{\mathscr{B}}t^{\alpha} \right), \quad t\in \mathbb{R_{+}}
				\label{3'}. 
			\end{align}
	Next, we claim that for $\mathscr{S}_{\alpha,n}(t;\mathscr{A})x$ for $Re(\lambda)>\omega$, $x \in \mathcal{X}$, and $n\in \mathbb{N}$:
	\begin{equation}\label{bir}
		\int_{0}^{\infty}e^{-\lambda t}\mathscr{S}_{\alpha,n}(t;\mathscr{A})x\mathrm{d}t=\mathcal{R}(\lambda^{\alpha};\mathscr{A})\int_{0}^{\infty}e^{-\lambda t}\mathscr{B}\mathscr{S}_{\alpha,n-1}(t;\mathscr{A})x\mathrm{d}t.
	\end{equation}
 The integrals in \eqref{bir} exist by \eqref{formula-1}. With the help of well-known Fubini's theorem for iterated integrals and the relation \eqref{sinus}, we verify that
 \allowdisplaybreaks
 \begin{align*}
 \int_{0}^{\infty}e^{-\lambda t}\mathscr{S}_{\alpha,n}(t;\mathscr{A})x\mathrm{d}t&=
 \int_{0}^{\infty}e^{-\lambda t}\int_{0}^{t}\mathscr{T}_{\alpha}(t-s;\mathscr{A})\mathscr{B}\mathscr{S}_{\alpha,n-1}(s;\mathscr{A})x\mathrm{d}s\mathrm{d}t\\
 &=\int_{0}^{\infty}\int_{0}^{t}e^{-\lambda t}\mathscr{T}_{\alpha}(t-s;\mathscr{A})\mathscr{B}\mathscr{S}_{\alpha,n-1}(s;\mathscr{A})x\mathrm{d}s\mathrm{d}t\\&=\int_{0}^{\infty}\int_{s}^{\infty}e^{-\lambda t}\mathscr{T}_{\alpha}(t-s;\mathscr{A})\mathscr{B}\mathscr{S}_{\alpha,n-1}(s;\mathscr{A})x\mathrm{d}t\mathrm{d}s\\
 &=\int_{0}^{\infty}\int_{0}^{\infty}e^{-\lambda (s+t)}\mathscr{T}_{\alpha}(t;\mathscr{A})\mathscr{B}\mathscr{S}_{\alpha,n-1}(s;\mathscr{A})x\mathrm{d}t\mathrm{d}s\\&=\int_{0}^{\infty}e^{-\lambda s}\int_{0}^{\infty}e^{-\lambda t}\mathscr{T}_{\alpha}(t;\mathscr{A})\mathscr{B}\mathscr{S}_{\alpha,n-1}(s;\mathscr{A})x\mathrm{d}t\mathrm{d}s\\
 &=\int_{0}^{\infty}e^{-\lambda s}\mathcal{R}(\lambda^{\alpha};\mathscr{A})\mathscr{B}\mathscr{S}_{\alpha,n-1}(s;\mathscr{A})x\mathrm{d}s\\&=\mathcal{R}(\lambda^{\alpha};\mathscr{A})\int_{0}^{\infty}e^{-\lambda t}\mathscr{B}\mathscr{S}_{\alpha,n-1}(t;\mathscr{A})x\mathrm{d}t.
\end{align*}
Then, we claim that for $\mathscr{S}_{\alpha}(t;\mathscr{A}+\mathscr{B})x$, $t\in \mathbb{R_{+}}$, for sufficiently large $\lambda$ and $x\in \mathcal{X}$,
\begin{equation}\label{iki}
	\int_{0}^{\infty}e^{-\lambda t}\mathscr{S}_{\alpha}(t;\mathscr{A}+\mathscr{B})x\mathrm{d}t=\mathcal{R}(\lambda^{\alpha};\mathscr{A}+\mathscr{B})x.
\end{equation}
From \eqref{bir}, we infer that for $n\in \mathbb{N}$, $x\in \mathcal{X}$, and $Re(\lambda)>\omega$:

 \begin{align*}
	\int_{0}^{\infty}e^{-\lambda t}\mathscr{S}_{\alpha,n}(t;\mathscr{A})x\mathrm{d}t&=\mathcal{R}(\lambda^{\alpha};\mathscr{A})\mathscr{B}\int_{0}^{\infty}e^{-\lambda t}\mathscr{S}_{\alpha,n-1}(t;\mathscr{A})x\mathrm{d}t\\&=\mathcal{R}(\lambda^{\alpha};\mathscr{A})\mathscr{B}\Big[\mathcal{R}(\lambda^{\alpha};\mathscr{A})\int_{0}^{\infty}e^{-\lambda t}\mathscr{B}\mathscr{S}_{\alpha,n-2}(t;\mathscr{A})x\mathrm{d}t\Big]\\&=\mathcal{R}(\lambda^{\alpha};\mathscr{A})\Big[\mathscr{B}\mathcal{R}(\lambda^{\alpha};\mathscr{A})]\mathscr{B}\int_{0}^{\infty}e^{-\lambda t}S_{\alpha,n-2}(t;\mathscr{A})x\mathrm{d}t\\&=\ldots=\mathcal{R}(\lambda^{\alpha};\mathscr{A})\Big[\mathscr{B}\mathcal{R}(\lambda^{\alpha};\mathscr{A})]^{n}x.
\end{align*}

Then \eqref{iki} follows from \eqref{formula-1}, \eqref{3} and Lemma \ref{lem}.

In a similar way, by using the following formula \eqref{bir-1} and relation \eqref{kosinus} for $Re(\lambda)>\omega$, $x\in \mathcal{X}$ and $n\in \mathbb{N}$:
%We next claim that for $C_{\alpha,n}(t;\mathscr{A})x$ for $\lambda>\omega$, $x \in X$, and $n\in \mathbb{N}:$
\begin{equation}\label{bir-1}
	\int_{0}^{\infty}e^{-\lambda t}\mathscr{C}_{\alpha,n}(t;\mathscr{A})x\mathrm{d}t=\lambda^{\alpha-1}\mathcal{R}(\lambda^{\alpha};\mathscr{A})\int_{0}^{\infty}e^{-\lambda t}\mathscr{C}_{\alpha,n-1}(t;\mathscr{A})x\mathrm{d}t,
\end{equation}
%The integrals in \eqref{bir-1} exist by \eqref{formula-2}. Equality \eqref{bir-1} holds, because
%\begin{align*}
	%\int_{0}^{\infty}e^{-\lambda t}C_{\alpha,n}(t;\mathscr{A})x\mathrm{d}t&=
	%\int_{0}^{\infty}e^{-\lambda t}\int_{0}^{t}\mathscr{C}_{\alpha}(t-s;\mathscr{A})BS_{\alpha,n-1}(s;\mathscr{A})x\mathrm{d}s\mathrm{d}t\\
	%&=\int_{0}^{\infty}\int_{0}^{t}e^{-\lambda t}\mathscr{C}_{\alpha}(t-s;\mathscr{A})BS_{\alpha,n-1}(s;\mathscr{A})x\mathrm{d}s\mathrm{d}t\\
	%&=\int_{0}^{\infty}\int_{s}^{\infty}e^{-\lambda t}\mathscr{C}_{\alpha}(t-s;\mathscr{A})BS_{\alpha,n-1}(s;\mathscr{A})x\mathrm{d}t\mathrm{d}s\\
	%&=\int_{0}^{\infty}\int_{0}^{\infty}e^{-\lambda (s+t)}\mathscr{C}_{\alpha}(t;\mathscr{A})BS_{\alpha,n-1}(s;\mathscr{A})x\mathrm{d}t\mathrm{d}s\\
	%&=\int_{0}^{\infty}e^{-\lambda s}\int_{0}^{\infty}e^{-\lambda t}\mathscr{C}_{\alpha}(t;\mathscr{A})BS_{\alpha,n-1}(s;\mathscr{A})x\mathrm{d}t\mathrm{d}s\\
	%&=\int_{0}^{\infty}e^{-\lambda s}\lambda^{\alpha-1}R(\lambda^{\alpha};\mathscr{A})BS_{\alpha,n-1}(s;\mathscr{A})x\mathrm{d}s\\
	%&=\lambda^{\alpha-1}R(\lambda^{\alpha};\mathscr{A})\int_{0}^{\infty}e^{-\lambda s}BS_{\alpha,n-1}(s;\mathscr{A})x\mathrm{d}s
%\end{align*}
%Next, we claim that for $\mathscr{C}_{\alpha}(t;\mathscr{A}+\mathscr{B})x$, $t\in \mathbb{R_{+}}$, for sufficiently large $\lambda$ and $x\in X$,
we can derive the corresponding result for fractional strongly continuous cosine families: 
\begin{equation}\label{iki-2}
\int_{0}^{\infty}e^{-\lambda t}\mathscr{C}_{\alpha}(t;\mathscr{A}+\mathscr{B})x\mathrm{d}t=\lambda^{\alpha-1}\mathcal{R}(\lambda^{\alpha};\mathscr{A}+\mathscr{B})x.
\end{equation}
Then, \eqref{iki-2} follows from \eqref{formula-2}, \eqref{3'} and Corollary \ref{corol}.
%From \eqref{bir-1} we infer that for $n\in \mathbb{N}$, $x\in X$, and $\lambda>\omega$, we have:
%\begin{align*}
%\int_{0}^{\infty}e^{-\lambda t}C_{\alpha,n}(t;\mathscr{A})x\mathrm{d}t&=\lambda^{\alpha-1}R(\lambda^{\alpha};\mathscr{A})B\int_{0}^{\infty}e^{-\lambda s}C_{\alpha,n-1}(s;\mathscr{A})x\mathrm{d}s\\
%&=\lambda^{\alpha-1}R(\lambda^{\alpha};\mathscr{A})B\Big[R(\lambda^{\alpha};\mathscr{A})\int_{0}^{\infty}e^{-\lambda s}C_{\alpha,n-2}(s;\mathscr{A})x\mathrm{d}s\Big]\\
%&=\lambda^{\alpha-1}R(\lambda^{\alpha};\mathscr{A})\Big[BR(\lambda^{\alpha};\mathscr{A})]B\int_{0}^{\infty}e^{-\lambda s}C_{\alpha,n-2}(s;\mathscr{A})x\mathrm{d}s\\
%&=\ldots=R(\lambda^{\alpha};\mathscr{A})\Big[BR(\lambda^{\alpha};\mathscr{A})]^{n}x.
%\end{align*}
%Then, from \eqref{iki-2} follows from \eqref{formula-2}, \eqref{3'} and Lemma \ref{corol}.
The proof is complete.
\end{proof}

\begin{remark}
Our results coincide with the classical ones whenever $\alpha=2$ \cite{TW3,lutz}.	
It should be note that in the case of $\alpha=2$, $\mathscr{T}_{\alpha}(t;\mathscr{A})$ and $\mathscr{S}_{\alpha}(t;\mathscr{A})$ coincide with the strongly continuous sine function $\mathscr{S}(t;\mathscr{A})$ and correspondingly, $\mathscr{C}_{\alpha}(t;\mathscr{A})$ coincide with the strongly continuous cosine function $\mathscr{C}(t;\mathscr{A})$. Furthermore, for $\alpha=2$, two parameter Mittag-Leffler type functions are converting to the hyperbolic sine and cosine functions, respectively:
\begin{align}
	tE_{2,2}(M\norm{\mathscr{B}}t^{2})=\sum_{k=0}^{\infty}\frac{M^{k}\norm{\mathscr{B}}^{k}t^{2k+1}}{(2k+1)!}=\frac{1}{\sqrt{M\norm{\mathscr{B}}}}\sinh\Big(\sqrt{M\norm{\mathscr{B}}}t\Big), \quad t \in \mathbb{R},\\
	E_{2,1}(M\norm{\mathscr{B}}t^{2})=\sum_{k=0}^{\infty}\frac{M^{k}\norm{\mathscr{B}}^{k}t^{2k}}{(2k)!}=\cosh\Big(\sqrt{M\norm{\mathscr{B}}}t\Big), \quad t \in \mathbb{R}.
\end{align}

\begin{theorem}\label{class}
	Let $\mathscr{A}$ be the infinitesimal generator of the strongly continuous cosine family $\left\lbrace \mathscr{C}(t;\mathscr{A}): t \in \mathbb{R}\right\rbrace $ with  $\|\mathscr{C}(t;\mathscr{A})\|\leq Me^{\omega |t|}$ for all $t \in \mathbb{R}$, and let $\left\lbrace \mathscr{S}(t;\mathscr{A}): t \in \mathbb{R}\right\rbrace $ denote the strongly continuous sine family associated with $\mathscr{C}$. Furthermore, suppose that $\mathscr{B}\in \mathcal{L}(\mathcal{X})$.  Then, strongly continuous families of linear bounded operators $\mathscr{C}(\cdot;\mathscr{A}+\mathscr{B}),\mathscr{S}(\cdot;\mathscr{A}+\mathscr{B})\in \mathcal{L}(\mathcal{X})$ generated by $\mathscr{A}+\mathscr{B}$ (defined on $\mathcal{D}(\mathscr{A})$) can be represented by the series expansion:
	\begin{align*}
		&\mathscr{C}(t;\mathscr{A}+\mathscr{B})x\coloneqq \sum_{n=0}^{\infty}\mathscr{C}_{n}(t,A)x, \quad x \in \mathcal{X},\\
		&\mathscr{S}(t;\mathscr{A}+\mathscr{B})x\coloneqq \sum_{n=0}^{\infty}\mathscr{S}_{n}(t;\mathscr{A})x, \quad x\in \mathcal{X},
	\end{align*}
	where for $t\in \mathbb{R}$, $x\in \mathcal{X}$ and $n \in \mathbb{N}$,  
	\begin{align*}
		&\mathscr{C}_{0}(t;\mathscr{A})x\coloneqq \mathscr{C}(t;\mathscr{A})x,\quad \mathscr{S}_{0}(t;\mathscr{A})x\coloneqq \mathscr{S}(t;\mathscr{A})x,\nonumber\\
		&\mathscr{C}_{n}(t;\mathscr{A})x\coloneqq \int_{0}^{t}\mathscr{C}(t-s;\mathscr{A})\mathscr{B}\mathscr{S}_{n-1}(s;\mathscr{A})x\mathrm{d}s,\\
		&\mathscr{S}_{n}(t;\mathscr{A})x\coloneqq \int_{0}^{t}\mathscr{S}(t-s;\mathscr{A})\mathscr{B}\mathscr{S}_{n-1}(s;\mathscr{A})x\mathrm{d}s.
	\end{align*}
\end{theorem}	
\end{remark}

We conclude this paper with an open question concerning with the perturbation of infinitesimal generators of the fractional strongly continuous cosine families:	
\begin{itemize}
	\item If $\mathscr{A}$ is the infinitesimal generator of a fractional strongly continuous cosine (sine) family and $\mathscr{B}$ is a closed linear operator in a Banach space $\mathcal{X}$, is $\mathscr{A}+\mathscr{B}$ the infinitesimal generator of a fractional strongly continuous cosine (sine) family?
\end{itemize}


\begin{thebibliography}{99}
		
	\bibitem{philips}
	R.S. Phillips, Perturbation theory for semigroups of linear operators, Trans. Am. Math. Soc., 74 (1954) 199-221.
		
	\bibitem{TW3}
	C.C. Travis, G.F. Webb, Perturbation of strongly continuous cosine family generators, Colloquium Mathematicae, 45(2) (1981) 277-285.	
		
	\bibitem{lutz}
	D. Lutz, On bounded time-dependent perturbations of operator cosine functions, Aequationes Mathematicae, 23 (1981) 197-203.
	
	
	\bibitem{bazhlekova2}
	E. Bazhlekova, Perturbation properties for abstract evolution equations of fractional order, Fract. Cal. Appl. Anal., 2(4) (1999) 359-366.
	
	\bibitem{A-H-M}
	A. Ahmadova, I.T. Huseynov, N.I. Mahmudov, Perturbation theory for fractional evolution equations in a Banach space, preprint at 	arXiv:2108.13188.
		
	\bibitem{engel}
	K.-J. Engel, R. Nagel, One-Parameter semigroups for linear evolution equations, vol. 194, Graduate Texts in Mathematics, Springer-Verlag, New York, 2000.	
		
		
		
		
	\bibitem{kexue1}
	K. Li, Fractional order semilinear Volterra integrodifferential equations in Banach spaces, Topol. Methods Nonlinear Anal., 47(2) (2016) 439-455.
		
		
		
	
	
		
	
		
		
		
		
		
		
	
		
		
		
		
	
		
		
	
		
		
		
	
		
		
		
	\end{thebibliography}
\end{document}